\documentclass[12pt]{amsart}
\usepackage{amssymb}
\usepackage{hyperref}
\usepackage{tabularx}
\usepackage{booktabs}
\usepackage{caption}

\newtheorem{thm}{Theorem}[section]
\newtheorem{lem}[thm]{Lemma}
\newtheorem{cor}[thm]{Corollary}
\newtheorem{prop}[thm]{Proposition}
\newtheorem{ex}[thm]{Example}

\newtheorem*{prob*}{Open problem}

\theoremstyle{definition}

\newtheorem{defi}[thm]{Definition}

\theoremstyle{remark}

\newtheorem{rem}[thm]{Remark}
\newtheorem*{rem*}{Remark}

\DeclareMathOperator{\id}{id}

\DeclareMathOperator{\rad}{rad}
\DeclareMathOperator{\Hom}{Hom}

\newcommand{\kringel}{\mathbin{\raise1pt\hbox{$\scriptstyle\circ$}}}
\newcommand{\pkt}{\mathbin{\raise0pt\hbox{$\scriptstyle\bullet$}}}

\newcommand{\C}{\mathbb{C}}

\newcommand{\ad}{{\rm ad}}

\newcommand{\Ann}{{\rm Ann}}
\newcommand{\End}{{\rm End}}
\newcommand{\Der}{{\rm Der}}

\newcommand{\nil}{\mathop{\rm nil}}

\newcommand{\La}{\mathfrak{a}}
\newcommand{\Lb}{\mathfrak{b}}

\newcommand{\Lg}{\mathfrak{g}}
\newcommand{\Lh}{\mathfrak{h}}

\newcommand{\Ll}{\mathfrak{l}}
\newcommand{\Ln}{\mathfrak{n}}

\newcommand{\Lp}{\mathfrak{p}}
\newcommand{\Lq}{\mathfrak{q}}

\newcommand{\Ls}{\mathfrak{s}}
\newcommand{\Lt}{\mathfrak{t}}

\newcommand{\Lz}{\mathfrak{z}}

\newcommand{\LL}{\mathfrak{L}}

\newcommand{\im}{\mathop{\rm im}}

\newcommand{\fix}{{\rm fix}}

\newcommand{\al}{\alpha}
\newcommand{\be}{\beta}
\newcommand{\ga}{\gamma}
\newcommand{\de}{\delta}

\newcommand{\la}{\lambda}

\newcommand{\ov}{\overline}

\newcommand{\ra}{\rightarrow}

\renewcommand{\phi}{\varphi}

\begin{document}


\title[Post-Lie Algebra Structures]{Commutative post-Lie algebra structures on Lie algebras}

\author[D. Burde]{Dietrich Burde}
\author[W. Moens]{Wolfgang Alexander Moens}
\address{Fakult\"at f\"ur Mathematik\\
Universit\"at Wien\\
  Oskar-Morgenstern-Platz 1\\
  1090 Wien \\
  Austria}
\email{dietrich.burde@univie.ac.at}
\address{Fakult\"at f\"ur Mathematik\\
Universit\"at Wien\\
  Oskar-Morgenstern-Platz 1\\
  1090 Wien \\
  Austria}
\email{wolfgang.moens@univie.ac.at}

\date{\today}

\thanks{The authors acknowledge support by the Austrian Science Foundation FWF, grant P28079 and grant J3371.}

\begin{abstract}
We show that any CPA-structure (commutative post-Lie algebra structure) on a perfect Lie algebra 
is trivial. Furthermore we give a general decomposition of inner CPA-structures, and classify all CPA-structures
on parabolic subalgebras of simple Lie algebras.
 \end{abstract}

\maketitle

\section{Introduction}

Post-Lie algebras have been introduced by Valette in connection with the homology of partition posets 
and the study of Koszul operads \cite{VAL}. Loday \cite{LOD} studied pre-Lie algebras and post-Lie
algebras within the context of algebraic operad triples. We rediscovered
post-Lie algebras as a natural common generalization of pre-Lie algebras \cite{HEL,KIM,SEG,BU5,BU19,BU24} and 
LR-algebras \cite{BU34, BU38} in the geometric context of nil-affine actions of Lie groups. We then studied 
post-Lie algebra structures in general, motivated by the importance of pre-Lie algebras in geometry, and 
in connection with generalized Lie algebra derivations \cite{BU41,BU44,BU51}. In particular, the existence 
question of post-Lie algebra structures on a given pair of Lie algebras turned out to be very interesting and
quite challenging. But even if existence is clear the question remains how many structures are possible.
In \cite{BU51} we introduced a special class of post-Lie algebra structures, namely
{\em commutative} ones. We conjectured that any commutative post-Lie algebra structure, in short CPA-structure,
on a complex, perfect Lie algebra is {\em trivial}. For several special cases we already proved the conjecture
in \cite{BU51}, but the general case remained open. One main result of this article here is a full proof of 
this conjecture, see Theorem $\ref{3.3}$. Furthermore we also study inner CPA-structures and give a 
classification of CPA-structures on parabolic subalgebras of semisimple Lie algebras. \\
In section $2$ we study ideals of CPA-structures, non-degenerate and inner CPA-structures. In particular we show 
that any CPA-structure on a complete Lie algebra is inner. We give a general decomposition of inner CPA-structures,
see Theorem $\ref{2.14}$. This implies, among other things, that any Lie algebra $\Lg$ admitting a non-degenerate
inner CPA-structure is metabelian, i.e., satisfies $[[\Lg,\Lg],[\Lg,\Lg]]=0$. \\
In section $3$ we prove the above conjecture and generalize the result to perfect subalgebras of arbitrary
Lie algebras in Theorem $\ref{3.4}$. This also implies that any Lie algebra admitting a non-degenerate CPA-product 
is solvable. Conversely we show that any non-trivial solvable Lie algebra admits a non-trivial CPA-product. \\
In section $4$ we classify all CPA-structures on parabolic subalgebras of simple Lie algebras in Theorem $\ref{4.7}$.
We obtain an explicit description of these products for standard Borel subalgebras of simple Lie algebras.

\section{Preliminaries}

Let $K$ always denote a field of characteristic zero. Post-Lie algebra structures on pairs of 
Lie algebras $(\Lg,\Ln)$ over $K$ are defined as follows \cite{BU41}:

\begin{defi}\label{pls}
Let $\Lg=(V, [\, ,])$ and $\Ln=(V, \{\, ,\})$ be two Lie brackets on a vector space $V$ over 
$K$. A {\it post-Lie algebra structure} on the pair $(\Lg,\Ln)$ is a $K$-bilinear product
$x\cdot y$ satisfying the identities:
\begin{align}
x\cdot y -y\cdot x & = [x,y]-\{x,y\} \label{post1}\\
[x,y]\cdot z & = x\cdot (y\cdot z) -y\cdot (x\cdot z) \label{post2}\\
x\cdot \{y,z\} & = \{x\cdot y,z\}+\{y,x\cdot z\} \label{post3}
\end{align}
for all $x,y,z \in V$.
\end{defi}

Define by  $L(x)(y)=x\cdot y$ and $R(x)(y)=y\cdot x$ the left respectively right multiplication 
operators of the algebra $A=(V,\cdot)$. By \eqref{post3}, all $L(x)$ are derivations of the Lie 
algebra $(V,\{,\})$. Moreover, by \eqref{post2}, the left multiplication
\[
L\colon \Lg\ra \Der(\Ln)\subseteq \End (V),\; x\mapsto L(x)
\]
is a linear representation of $\Lg$. A particular case of a post-Lie algebra structure arises 
if the algebra $A=(V,\cdot)$ is {\it commutative}, i.e., if $x\cdot y=y\cdot x$ is satisfied for all 
$x,y\in V$. Then the two Lie brackets $[x,y]=\{x,y\}$ coincide, and we obtain a commutative algebra 
structure on $V$ associated with only one Lie algebra \cite{BU51}.

\begin{defi}\label{cpa}
A {\it commutative post-Lie algebra structure}, or {\em CPA-structure} on a Lie algebra $\Lg$ 
is a $K$-bilinear product $x\cdot y$ satisfying the identities:
\begin{align}
x\cdot y & =y\cdot x \label{com4}\\
[x,y]\cdot z & = x\cdot (y\cdot z) -y\cdot (x\cdot z)\label{com5} \\
x\cdot [y,z] & = [x\cdot y,z]+[y,x\cdot z] \label{com6}
\end{align}
for all $x,y,z \in V$. The associated algebra $A=(V,\cdot)$ is called a CPA.
\end{defi}

There is always the {\it trivial} CPA-structure on $\Lg$, given by $x\cdot y=0$ for all $x,y\in \Lg$. 
However, in general it is not obvious whether or not a given Lie algebra admits a non-trivial CPA-structure.
For abelian Lie algebras, CPA-structures correspond to commutative associative 
algebras:

\begin{ex}
Suppose that $(A,\cdot)$ is a CPA-structure on an abelian Lie algebra $\Lg$. Then $A$ is 
commutative and associative.
\end{ex}

Indeed, using \eqref{com4}, \eqref{com5} and $[x,y]=0$ we have
\[
x\cdot (z\cdot y)=x\cdot(y\cdot z)=y\cdot (x\cdot z) =(x\cdot z)\cdot y
\]
for all $x,y,z\in \Lg$. \\
It is easy to see that there are examples only admitting trivial CPA-structures:

\begin{ex}
Every CPA-structure on $\Ls\Ll_2(K)$ is trivial.
\end{ex}

This follows from a direct computation, but also holds true more generally for every semisimple Lie algebra,
see Proposition $\ref{3.1}$.
One main aim of this paper is to show that this is even true for all {\em perfect} Lie algebras, see
Theorem $\ref{3.3}$.

\begin{defi}
A CPA-structure  $(A,\cdot)$ on $\Lg$ is called {\em nondegenerate} if the annihilator 
$$
\Ann_A=\ker(L)=\{ x\in \Lg\mid L(x)=0\}
$$
is trivial.
\end{defi}

Note that $\Ann_A$ is an ideal of the CPA as well as an ideal of the Lie algebra. Here
a subspace $I$ of $V$ is an algebra ideal if $A\cdot I\subseteq I$, and a Lie algebra ideal
if $[\Lg,I]\subseteq I$. An {\em ideal} is defined to be an ideal for both $A$ and $\Lg$.
Let $[x_1,\ldots ,x_n]:=[x_1,[x_2,[x_3,\ldots ,x_n]]]\cdots ]$ and 
$I^{[n]}:=[I,[I,[I, \cdots ]]]\cdots ]$ for an ideal $I$.

\begin{prop}\label{2.6}
Supoose that $(A,\cdot)$ is a CPA-structure on $\Lg$. Then there exists an ideal $I_{\infty}$ 
such that

\begin{itemize}
\item[$(1)$] $I_{\infty}^{[k]} \subseteq \Ann_A\subseteq I_{\infty}$ for all $k$ large enough.
\item[$(2)$] The CPA-structure on $\Lg/I_{\infty}$ is nondegenerate.
\end{itemize}
\end{prop}

\begin{proof}
Define an ascending chain of ideals $I_n$ by $I_0=0$ and $I_n=\{ x\in A\mid x\cdot A\subseteq I_{n-1}\}$
for $n\ge 1$. We have $I_1=\Ann_L$ and each $I_n$ is indeed an ideal because of 
$I_n\cdot A\subseteq I_{n-1}\subseteq I_n$, and
\begin{align*}
[I_n,\Lg]\cdot A  & \subseteq I_n\cdot (\Lg\cdot A)+\Lg\cdot (I_n\cdot A)\\
                  &  \subseteq I_n\cdot A +\Lg\cdot I_{n-1} \\
                  & \subseteq I_{n-1}.
\end{align*}
So for $x\in [I_n,\Lg]$ and $a\in A$ we have $x\cdot a\in I_{n-1}$, hence $x\in I_n$.
Since $\Lg$ is finite-dimensional, this chain stabilizes, i.e., there exists a minimal $k$ such that 
$I_k=I_{\ell}$ for all $\ell\ge k$. Then define $I_{\infty}:=I_k$.
By construction we have $A\cdot I_1=0$, $A\cdot (A\cdot I_2)\subseteq A\cdot I_1=0$,
etc., so that right-associative products in $I_{\infty}$ of length at least $k+1$ vanish. Using
\eqref{com5} we have
\begin{align*}
[x_1,\ldots, x_{n-1},x_n]\cdot z & =[x_1,\ldots ,x_{n-1}]\cdot (x_n\cdot z)-x_n\cdot ([x_1,\ldots, x_{n-1}]\cdot z)
\end{align*}
for all $x,y,z\in V$. 
By induction we see that the elements $[x_1,\ldots ,x_n]\cdot z$ are spanned by the 
right-associative elements $x_{\pi(1)}\cdot x_{\pi(2)}\cdots x_{\pi(n)}\cdot z$, where $\pi$ runs over all
permutations in $S_n$. This yields $I_{\infty}^{[k+1]}\cdot \Lg=0$, and hence 
$I_{\infty}^{[k+1]}\subseteq \Ann_A$. We also have $\Ann_A=I_1\subseteq I_{\infty}$. 
Furthermore $x\cdot \Lg\subseteq I_{\infty}$ implies $x\in I_{\infty}$, so that the induced CPA-structure 
on  $\Lg/I_{\infty}$ is nondegenerate.
Note that $I_{\infty}$ is in fact the minimal ideal with this property.
\end{proof}

\begin{defi}
A CPA-structure on $\Lg$ is called {\em weakly inner}, if there is a $\phi\in\End(V)$ such that the algebra
product is given by
$$
x\cdot y=[\phi(x),y].
$$
It is called {\em inner}, if in addition $\phi$ is a Lie algebra homomorphism, i.e., $\phi \in \End(\Lg)$.
\end{defi}

In terms of operators this means that we have $L(x)=\ad (\phi(x))$ for all $x\in V$. We have
$\ker(\phi)\subseteq \ker(L)$ with equality for $Z(\Lg)=0$.

\begin{lem}
Let $\Lg$ be a Lie algebra with trivial center.  Then any weakly inner CPA-structure on $\Lg$
is inner.
\end{lem}

\begin{proof}
A product $x\cdot y=[\phi(x),y]$ with some $\phi\in\End(V)$ defines a CPA-structure 
on $\Lg$, if and only if
\begin{align*}
[\phi(x),y] & = [\phi(y),x] \\
[[\phi(x),\phi(y)],z] & = [\phi([x,y]),z] 
\end{align*}
for all $x,y,z\in \Lg$. In case that $Z(\Lg)=0$ the last condition says that $\phi$ is a 
Lie algebra homomorphism. 
\end{proof}

\begin{cor}\label{2.9}
Let $\Lg$ be a complete Lie algebra. Then any CPA-structure on $\Lg$ is inner.
\end{cor}

\begin{proof}
By definition we have $\Der(\Lg)=\ad(\Lg)$ and $Z(\Lg)=0$. Hence $L(x)\in \Der(\Lg)$ implies that 
$L(x)=\ad(\phi(x))$ for some $\phi\in \End(\Lg)$. 
\end{proof}

In general not all CPA-structures on a Lie algebra are inner or weakly inner. This is trivially the case for abelian
Lie algebras, which do admit nonzero CPA-structures, which cannot be weakly inner. The Heisenberg Lie algebra
$\Lh_1 =\langle e_1,e_2,e_3 \mid [e_1,e_2]=e_3\rangle $ admits a family $A(\mu)$ of CPA-structures
given by  $e_1\cdot e_1=e_2$, $e_1\cdot e_2=e_2\cdot e_1=\mu e_3$ for $\mu\in K$, see Proposition $6.3$ in
\cite{BU51}:

\begin{ex}
The CPA-product $A(\mu)$ on the Heisenberg Lie algebra $\Lh_1$ is not weakly inner.
\end{ex}

Indeed, all $\ad(\phi(x))$ map $\Lh_1$ into its center, whereas $L(e_1)$ does not.
Hence $L(x)=\ad(\phi(x))$ cannot hold for all $x\in \Lh_1$. 

\begin{lem}\label{2.11}
Suppose that $(A,\cdot)$ is an inner CPA-structure on $\Lg$. Then the ascending chain of ideals 
$I_n$ is invariant under $\phi$, and all Lie algebra ideals 
of $\Lg$ are ideals of $A$. Conversely, if the structure is nondegenerate, all ideals of $A$ are 
Lie algebra ideals. 
\end{lem}

\begin{proof} 
Let $I$ be a Lie algebra ideal. Then $\Lg\cdot I=[\phi(\Lg),I]\subseteq [\Lg,I]\subseteq I$.
Conversely, let $I$ be an algebra ideal and $(A,\cdot)$ be nondegenerate, given by $x\cdot y=[\phi(x),y]$
with $\phi$ being invertible. Then $\phi(\Lg)=\Lg$, so that
\[
[\Lg,I]=[\phi(\Lg),I]=\Lg\cdot I\subseteq I.
\]
The ideals $I_n$ were defined by  $I_0=0$ and  $I_n=\{ x\in A\mid x\cdot A\subseteq I_{n-1}\}$ for $n\ge 1$. 
Clearly $\phi(I_0)=I_0$. Using induction we obtain
\begin{align*}
\Lg\cdot \phi(I_n) & =[\phi(\Lg),\phi (I_n)] \\
                  & =\phi([\Lg,I_n])\\
                  & \subseteq \phi(I_{n-1}) \subseteq I_{n-1}.
\end{align*}
Hence $\phi(I_n)\subseteq I_n$ for all $n$.
\end{proof}

\begin{lem}\label{2.12}
Suppose that $x\cdot y=[\phi(x),y]$ is an inner CPA-structure on a complex Lie algebra $\Lg$, and let
$\Lg=\bigoplus_{\al}\Lg_{\al}$ be the generalized eigenspace decomposition of $\Lg$
with respect to $\phi$. Then we have
\begin{align*}
[\Lg_{\al},\Lg_{\be}] & \subseteq \Lg_{\al\be},\\
[\Lg_{\al},\Lg_{\be}] & \neq 0 \text{ implies } \al+\be=0.
\end{align*}
\end{lem}

\begin{proof}
The first statement is well-known, so that we only need to prove the second one.
Using $[\phi(x),y]=-[x,\phi(y)]$ we obtain
\[
\phi([x,y])=[\phi(x),\phi(y)] =-[\phi^2(x),y].
\]
By induction on $k\ge 0$ this yields
\[
(\phi+\gamma\id)^k([x,y])=(-1)^k\cdot [(\phi^2-\ga\id)^k(x),y].
\]
The RHS vanishes for $\gamma:=\al^2$ and $k$ large enough, since if $\phi$
has a generalized eigenvector $x$ with generalized eigenvalue $\al$, then 
$\phi^2$ has  generalized eigenvalue $\al^2$ for $x$. This yields $[\Lg_{\al},\Lg_{\be}]\subseteq \Lg_{-\al^2}$,
and similarly $[\Lg_{\al},\Lg_{\be}]\subseteq \Lg_{-\be^2}$, hence
\[
[\Lg_{\al},\Lg_{\be}]\subseteq \Lg_{-\al^2}\cap \Lg_{\al\be}\cap  \Lg_{-\be^2}.
\]
If $[\Lg_{\al},\Lg_{\be}] \neq 0$ then all three spaces coincide, so that $-\be^2=\al\be=-\al^2$,
i.e., $\al+\be=0$.
\end{proof}

\begin{defi}
A CPA-structure on $\Lg$ is called {\em nil-inner}, if it can be written as $x\cdot y=[\phi (x),y]$
with a nilpotent Lie algebra homomorphism $\phi\in \End(\Lg)$.
\end{defi}

The trivial CPA-structure is an example of a nil-inner structure. We can now obtain a general decomposition
of complex inner CPA-structures.

\begin{thm}\label{2.14}
Let $\Lg$ be a complex Lie algebra and suppose that it admits an inner CPA-structure with 
$\phi\in \End(\Lg)$. Then $\Lg$ decomposes into the sum of $\phi$-invariant ideals
\[
\Lg=\Ln\oplus\Lh
\]

with the following properties:
\begin{itemize}
\item[$(1)$] $\phi_{\mid \Ln}$ is a nilpotent endomorphism of $\Ln$ such that the CPA-structure on $\Ln$ 
is nil-inner.
\item[$(2)$] $\phi_{\mid \Lh}$ is an automorphism of $\Lh$, and we have $[[\Lh,\Lh],[\Lh,\Lh]]=0$.
\end{itemize}
\end{thm}

\begin{proof}
Consider the eigenspace decomposition
\[
\Lg=\bigoplus_{\al}\Lg_{\al}=\Lg_0\oplus \bigoplus_{\al\neq 0}\Lg_{\al}
\]
of $\Lg$ with respect to the Lie algebra homomorphism $\phi$, with $\Ln=\Lg_0$ and 
$\Lh=\oplus_{\al\neq 0}\Lg_{\al}$. Both $\Ln$ and $\Lh$ are Lie ideals, and hence ideals by Lemma $\ref{2.11}$, since 
$[\Lg_{\al},\Lg_{\be}]\subseteq \Lg_{-\al^2}$ implies that $[\Ln,\Lg]\subseteq \Ln$ and $[\Lh,\Lg]\subseteq \Lh$.
Moreover, both $\Ln$ and $\Lh$ are invariant under $\phi$, so that the restrictions of $\phi$ to $\Ln$ and $\Lh$ are
well-defined. Clearly the restriction of $\phi$ to $\Ln$ is nilpotent,
and since all generalized eigenvalues of $\Lh$ are nonzero, the restriction of $\phi$ to $\Lh$ is an automorphism.
It remains to show that $\Lh$ is metabelian, i.e., to show that 
\[
[[\Lg_{\al},\Lg_{\be}],[\Lg_{\ga},\Lg_{\de}]]=0
\]
for all $\al,\be, \ga,\de \neq 0$. Suppose this is not the case. Then Lemma $\ref{2.12}$ yields
\begin{align*}
\al+\be & = 0,\\
\ga+\de & = 0, \\
\al\be+\ga\de & = 0.
\end{align*}
Setting $\be=-\al$, $\ga=\al i$ and $\de=-\al i$ the bracket takes the form 
$[[\Lg_{\al},\Lg_{-\al}],[\Lg_{\al i},\Lg_{-\al i}]]\neq 0$. We may apply the Jacobi identity here in two ways:
\begin{align*}
[[\Lg_{\al},\Lg_{-\al}],[\Lg_{\al i},\Lg_{-\al i}]] & \subseteq [\Lg_{\al i},[\Lg_{-\al i},[\Lg_{\al},\Lg_{-\al}]]]+
[\Lg_{-\al i},[[\Lg_{\al},\Lg_{-\al}],\Lg_{\al i}]], 
\end{align*}
and 
\begin{align*}
[[\Lg_{\al},\Lg_{-\al}],[\Lg_{\al i},\Lg_{-\al i}]] &  \subseteq [[\Lg_{-\al},[\Lg_{\al i},\Lg_{-\al i}]],\Lg_{\al}]+
[[[\Lg_{\al i},\Lg_{-\al i}],\Lg_{\al}],\Lg_{-\al}].
\end{align*}
In each case, at least one of the terms on the right hand side must be nonzero. The first case gives us
that either $0\neq [\Lg_{-\al i},[\Lg_{\al},\Lg_{-\al}]]\subseteq [\Lg_{-\al i},\Lg_{-\al^2}]$, i.e., that
$-\al i-\al^2=0$, or $0\neq [\Lg_{\al i},[\Lg_{\al},\Lg_{-\al}]]\subseteq [\Lg_{\al i},\Lg_{-\al^2}]$, i.e., that
$\al i-\al^2=0$. This means $\al=\pm i$. The second case gives us that either $0\neq [\Lg_{-\al },[\Lg_{\al i},\Lg_{-\al i}]]
\subseteq [\Lg_{-\al},\Lg_{\al^2}]$, i.e., that $\al^2-\al=0$, or $0\neq [\Lg_{\al },[\Lg_{\al i},\Lg_{-\al i}]]
\subseteq [\Lg_{\al},\Lg_{\al^2}]$, i.e., that $\al^2+\al=0$. This means $\al=\pm 1$. So we must have both
$\al=\pm i$ and $\al=\pm 1$, which is impossible.
\end{proof}

\begin{cor}\label{2.15}
Let $\Lg$ be a Lie algebra over $K$ admitting a non-degenerate inner CPA-structure. Then $\Lg$ is metabelian. 
\end{cor}

\begin{proof}
Complexifying $\Lg$ the above Theorem implies that $\Lg=\Ln\oplus \Lh$ and $\Lh$ is metabelian.
Since $\ker(\phi)\subseteq \ker(L)=0$ we have $\Ln=0$ and $\Lg=\Lh$. Now $\Lg$ is metabelian over $\C$ if
and only if $\Lg$ is metabelian over $K$.
\end{proof}

Let $\Lb$ be the standard Borel subalgebra of $\Ls\Ll_2(K)$ with basis $e_1=E_{12}$, $e_2=E_{11}-E_{22}$ and 
Lie bracket $[e_1,e_2]=-2e_1$. Here $E_{ij}$ denotes the matrix with entry $1$ at position $(i,j)$, and
entries $0$ otherwise.

\begin{ex}\label{2.16}
Every CPA-structure on the Borel subalgebra $\Lb$ of $\Ls\Ll_2(K)$ is inner, and is of the form
\[
L(e_1)=\begin{pmatrix} 0 & \al \cr 0 & 0 \end{pmatrix},\; 
L(e_2)=\begin{pmatrix} \al & \be \cr 0 & 0 \end{pmatrix}
\]
for $\al,\be \in K$ such that $\al(\al-2)=0$.
\end{ex}

Indeed, since $\Lb$ is complete, every CPA-structure on $\Lb$ is inner by Corollary $\ref{2.9}$.
A short computation shows that we have $L(x)=\ad (\phi(x)))$ with 
\[
\phi=\frac{1}{2}\begin{pmatrix} -\al & -\be \cr 0 & \al \end{pmatrix}
\]
and $\al(\al-2)=0$. Note that $\phi^2=0$ for $\al=0$, and $\phi^2=I$ for $\al=2$. The latter structure
is non-degenerate, so that $\Lb$ is metabelian according to Corollary $\ref{2.15}$. Of course, this is
obvious anyway.

\section{CPA-structures on perfect Lie algebras}

For this section we will assume that all Lie algebras are complex. We start with CPA-structures on semisimple 
Lie algebras, where we give another proof of Proposition $5.4$ and Corollary $5.5$ in \cite{BU51}, 
without using the structure results of \cite{LEL}:

\begin{prop}\label{3.1}
Any CPA-structure on a semisimple Lie algebra is trivial. Furthermore any
CPA-structure on a Lie algebra $\Lg$ satisfies $\Lg\cdot \Lg\subseteq \rad(\Lg)$.
\end{prop}

\begin{proof}
Let $(A,\cdot )$ be a CPA-structure on a semisimple Lie algebra $\Ls$. Then it is inner by
Corollary $\ref{2.9}$, i.e., given by $L(x)=\ad (\phi(x))$. We have $I_{\infty}^{[k]}\cdot \Ls=0$ 
for the ideal $I_{\infty}$ of Proposition $\ref{2.6}$. Since $I_{\infty}$ is invariant by Lemma $\ref{2.11}$
the quotient CPA-structure on $\Ls/I_{\infty}$ is also inner, and nondegenerate. Theorem $\ref{2.14}$ implies
that the Lie algebra  $\Ls/I_{\infty}$ is metabelian, hence solvable. Since $\Ls$ is perfect, any solvable
quotient is trivial. Hence we have $\Ls=I_{\infty}$ and $0=I_{\infty}^{[k]}\cdot \Ls=\Ls^{[k]}\cdot \Ls=\Ls\cdot \Ls$.
Hence the CPA-structure on $\Ls$ is trivial. The second part follows by considering the semisimple
quotient $\Lg/\rad(\Lg)$.
\end{proof}

\begin{lem}\label{3.2}
Let $\Ls$ be a semisimple Lie algebra. Then there exist Lie algebra generators $\{s_i \mid 1\le i\le m \}$ of 
$\Ls$ such that for every linear representation $\psi\colon \Ls\ra \Lg\Ll(V)$, all $\psi(s_i)$ are nilpotent.
\end{lem}

\begin{proof}
Let $\{e_i,f_i,h_i\mid 1\le i\le k\}$ be the Chevalley-Serre generators for $\Ls$. Each triple $(e_i,f_i,h_i)$
generates a subalgebra isomorphic to $\Ls\Ll_2(\C)$, and $\psi$ restricted to it is a representation. By the
classification of representations of $\Ls\Ll_2(\C)$ we know that $\psi(e_i)$ and $\psi(f_i)$ are nilpotent.
It follows that $\{e_i,f_i \mid 1\le i\le k\}$ is a set of generators for $\Ls$ such that all  $\psi(e_i)$ and
all  $\psi(f_i)$ are nilpotent.
\end{proof}

We are now able to prove Conjecture $5.21$ of \cite{BU51}.

\begin{thm}\label{3.3}
Any CPA-structure on a perfect Lie algebra $\Lg$ is trivial, i.e., satisfies $\Lg\cdot \Lg=0$.
\end{thm}

\begin{proof}
Let $\Lg$ be a perfect Lie algebra with Levi subalgebra $\Ls$ and solvable radical $\rad(\Lg)=\La$. 
We have $\Lg=\Ls\ltimes \La$. Denote by $\Der(\Lg,\La)$ the space of those derivations $D\in \Der(\Lg)$
satisfying $D(\Lg)\subseteq \La$. For the proof it is sufficient to show that 
$\Ls\cdot \Lg=0$, since $\Lg$ is perfect and hence $\Ls$ generates $\Lg$ as a Lie ideal by 
Lemma $5.15$ in \cite{BU51}. By Corollary $5.17$ in \cite{BU51} we may assume that $\La$ is {\em abelian}.
Decompose $\La$ into irreducible $\Ls$-modules $\La=\La_1\oplus \cdots \oplus \La_m$. By Proposition 
$\ref{3.1}$ we have $\Lg\cdot \Lg\subseteq \La$, i.e., $L(\Lg)(\Lg)\subseteq \La$, and hence 
$L(\Lg)\subseteq \Der(\Lg,\La)$. Lemma $5.18$ in \cite{BU51} gives a natural splitting
\[
\Der(\Lg,\La)=\Der_{\Ls}(\La) \ltimes Z^1(\Ls,\La),
\]
where 
\[
\Der_{\Ls}(\La)=\{ d\in \Der(\La)\mid \phi(x)d(a)=d(\phi(x)a) \; \forall\, x\in \Ls, a\in \La\}
\]
with $L(x)=\ad(\phi(x))$. Since $\Ls$ is semisimple, Whitehead's first Lemma implies that 
\begin{align*}
s\cdot \Lg & = \Lg\cdot s \\
           & = Z^1(\Ls,\La)(s) \\
           & = B^1(\Ls,\La)(s) \\
           & = [s,\La] \\
           & =[s,\La_1]+\cdots +[s,\La_m]
\end{align*}
for all $s\in \Ls$. On the other hand, we have the natural embeddings of vector spaces
\[
\Der_{\Ls}(\La)\subseteq \Hom_{\Ls}(\La)\subseteq \bigoplus_{i,j}\Hom(\La_i,\La_j).
\]
Hence for every $s\in \Ls$ there exist linear maps $f_{j,i}^s\in \Hom_{\Ls}(\La_i,\La_j)$ such that
\[
s\cdot v_i=\sum_{k=1}^m f_{k,i}^s(v_i)
\]
for all $v_i\in \La_i$, for every $i$. Altogether we obtain $f_{j,i}^s(\La_i)\subseteq [s,\La_j]$
for all $j,i\in \{1,\ldots ,m \}$. \\[0.2cm]
Suppose that $s\in\Ls$ is an element such that $[s,\La_j]\subsetneq \La_j$ for all $j$. Then Schur's Lemma
applied to the simple $\Ls$-modules $\La_j$ implies that $f_{j,i}^s=0$ for all $i,j$, so that $s\cdot \La=0$.
Now Lemma $\ref{3.2}$ applied to the linear representations $\psi_j=\ad_{\La_j}$ gives us a set of 
generators $\{s_1,\ldots ,s_k \}$ of $\Ls$ such that $\im(\psi_j(s_i))=[s_i,\La_j]\subsetneq \La_j$ for all $i,j$,
since all $\psi_j(s_i)$ are nilpotent. Thus we have $s_i\cdot \La=0$ for all $i$.
Since the $s_i$ generate $\Ls$ this means that $\Ls\cdot \La=0$, and hence
$L(\Ls)\subseteq Z^1(\Ls,\La)$. By Lemma $5.18$ in \cite{BU51} $Z^1(\Ls,\La)$ is abelian, so that $L(\Ls)$ is both
abelian and semisimple, hence trivial. We obtain $L(\Ls)=0$, so that $\Ls\cdot \Lg=0$ and the proof is finished.
\end{proof}

We can generalize the last result as follows.

\begin{thm}\label{3.4}
Let $\Lp$ be a perfect subalgebra of a Lie algebra $\Lg$. Then every CPA-structure on $\Lg$
satisfies $\Lp\cdot \Lg=0$.
\end{thm}

\begin{proof}
Let $\Lt$ be a Levi complement of $\Lp$. Then $\Lp\cdot \Lg=0$ if and only if
$\Lt\cdot \Lg=0$, again by Lemma $5.15$ in \cite{BU51} and the fact that for a set $X\subseteq \ker(L)$
the ideal in $\Lg$ generated by $X$ also lies in $\ker(L)$.
We have $\Lt\cdot \Lg\subseteq \Ls\cdot \Lg$ for some Levi complement $\Ls$ of $\Lg$. Hence it is enough to show that 
$\Ls\cdot \Lg=0$ for all Levi complements $\Ls$ of $\Lg$. Suppose first that $\Lg$ has no proper characteristic ideal
$I$ with $0\subsetneq I\subsetneq \rad(\Lg)$. Then $\rad(\Lg)$ is abelian, because otherwise $[\rad(\Lg),\rad(\Lg)]$ 
would be a proper characteristic ideal. Furthermore $\Lg$ is of the form $\Lg=\Ls\ltimes V^n$ with an irreducible 
$\Ls$-module. If $V$ is the trivial module, then $\Lg$ is reductive and we have $\Ls\cdot \Lg=0$ by Corollary $5.6$
of \cite{BU51}. Otherwise $\Lg=\Ls\ltimes V^n$ is perfect, and $\Ls\cdot \Lg=0$ by Theorem $\ref{3.3}$. \\
It remains to study the case where $\Lg$ admits a proper characteristic ideal  $0\subsetneq I\subsetneq \rad(\Lg)$.
Either we have $\Ls\cdot \Lg=0$ and we are done, or there exists a Lie algebra $\Lg$
with $\Ls\cdot \Lg\neq 0$. We may choose $\Lg$ so that it is of minimal dimension. 
By Proposition $\ref{3.1}$ we have $\Ls\cdot \Lg\subseteq \rad(\Lg)$, so that $\rad (\Lg)\neq 0$.
Since $\Ls$ is semisimple, the $\Lg$-module $\Lg$ given by the representation $x\mapsto L(x)$ has a $\Lg$-module 
complement $U$ with $\Lg=U\oplus \rad(\Lg)$. Using $\Ls\cdot \Lg\subseteq \rad(\Lg)$ we obtain $\Ls\cdot U=0$. 
Since $I$ is invariant under the $\Ls$-action, we have a module complement $K$ with $\rad(\Lg)=K\oplus I$.
The quotient algebra $\Lg/I$ then is isomorphic to $\Ls\ltimes K/I$, and the minimality of $\Lg$ implies
$\Ls\cdot \Lg\subseteq I$, so that $\Ls\cdot K\subseteq K\cap I=0$. We see that the Lie algebra $\Ls\ltimes I$
is closed under the CPA-product: since $I$ is a characteristic ideal of $\Lg$ we have $\Lg\cdot I\subseteq I$, and
\[
(\Ls\ltimes I)\cdot (\Ls\ltimes I) \subseteq \Ls\cdot \Lg+\Lg\cdot I \subseteq \Ls\ltimes I.
\]
Since $\Lg$ is minimal it follows that $\Ls\cdot I=0$, and 
\[
\Ls\cdot \Lg=\Ls\cdot (U+K+I)=\Ls\cdot U+\Ls\cdot K+\Ls\cdot I=0.
\]
This is a contradiction, and the proof is finished.
\end{proof}

\begin{cor}
Suppose that $\Lg$ admits a nondegenerate CPA-product. Then $\Lg$ is solvable.
\end{cor}

\begin{proof}
Let $\Ls$ be a Levi subalgebra of $\Lg$. Then $\Ls\cdot \Lg=0$ by Theorem $\ref{3.4}$, so that
$\Ls\subseteq \ker(L)=0$. Hence $\rad(\Lg)=\Lg$, and $\Lg$ is solvable.
\end{proof}

Since we know that a perfect Lie algebra only admits the trivial CPA-structure, it is natural
to ask for the converse. Given a non-perfect Lie algebra $\Lg$. Can we construct
a non-trivial CPA-structures on $\Lg$ ? The following example shows that this is not always possible.

\begin{ex}\label{3.6}
Let $\Lg$ denote the Lie subalgebra of $\Ls\Ll_3(\mathbb{C})$ of dimension $6$ with basis
\[
(e_1,\ldots, e_6)=(E_{12},E_{13},E_{21},E_{23},E_{11}-E_{22},E_{22}-E_{33}).
\]
Then $\Lg$ is not perfect and admits only the trivial CPA-product.
\end{ex}

The Lie brackets are given by
\begin{align*}
[e_1,e_3] & = e_5,\, [e_1,e_4] = e_2,\, [e_1,e_5] = -2e_1, \, [e_1,e_6] = e_1,\\
[e_2,e_3] & = -e_4,\, [e_2,e_5 ]= -e_2,\, [e_2,e_6] = -e_2, \, [e_3,e_5] = 2e_3,  \\
[e_3,e_6 ] & = -e_3,\, [e_4,e_5]= e_4,\, [e_4,e_6] = -2e_4.
\end{align*}
We have $\dim [\Lg,\Lg]=5$, so that $\Lg$ is not perfect. For a given CPA-structure
we know by Theorem $\ref{3.4}$ that $\Lp\cdot \Lg=0$ for the perfect subalgebra $\Lp=\text{span}\{e_1,\ldots, e_5\}$.
It is now easy to see that the CPA-product on $\Lg$ is trivial. \\[0.2cm]
On the other hand we will show that every solvable Lie algebra $\Lg$ admits a non-trivial CPA-structure.
Here we distinguish two cases, namely whether or not $\Lg$ has trivial center.

\begin{prop}
Let $\Lg$ be a solvable Lie algebra with trivial center. Then $\Lg$ admits a non-trivial nil-inner
CPA-structure.
\end{prop}

\begin{proof}
By Lie's theorem there exists a nonzero common eigenvector $v\in \Lg$ and a linear functional
$\la\colon \Lg\ra \C$ such that $[x,v]=\la(x)v$ for all $x\in \Lg$. We have 
\begin{align*}
\la([x,y])v & = [[x,y],v]\\
            & = [x,[y,v]]-[y,[x,v]] \\
            & = (\la(x)\la(y)-\la(y)\la(x))v \\
            & = 0.
\end{align*}
Hence $x\cdot_v y:=[x,[y,v]]=\la(x)\la(y)v$ defines a CPA-structure on $\Lg$. It is non-trivial, because
otherwise the center of $\Lg$ were non-trivial.
\end{proof}

\begin{prop}
Let $\Lg$ be a non-perfect Lie algebra with non-trivial center. Then $\Lg$ admits a non-trivial
CPA-product.
\end{prop}

\begin{proof}
Suppose first that $Z(\Lg)\cap [\Lg,\Lg]\neq 0$, and select a nonzero $z$ from it. Since $\Lg$ is not perfect
we may choose a $1$-codimensional ideal $I$ of $\Lg$ with $I\supseteq [\Lg,\Lg]$. Fix a basis $(e_2,\ldots ,e_n)$
for $I$ and a generator $e_1$ for the vector space complement of $I$ in $\Lg$. Then $\Lg$ is a semidirect
product $\C e_1\ltimes I$. Using the nonzero $z\in Z(\Lg)\cap [\Lg,\Lg]$ define a non-trivial CPA-structure 
on $\Lg$ by
\[
\left(\sum_{i=1}^n \al_ie_i\right)\cdot \left(\sum_{i=1}^n \be_ie_i\right):=\al_1\be_1 z.
\]
Now assume that $Z(\Lg)\cap [\Lg,\Lg]= 0$. Then $\Lg$ admits an abelian factor, because $Z(\Lg)\neq 0$.
So we can write $\Lg=\C e_1\oplus \Lh$ for some ideal $\Lh$ in $\Lg$. Let  $(e_2,\ldots ,e_n)$ be a basis of
$\Lh$ and define a non-trivial CPA-structure on $\Lg$ as before but replacing $z$ by $e_1$ on the RHS.
Note that in both cases the CPA-product is even associative.
\end{proof}

\begin{cor}
Let $\Lg$ be a non-trivial solvable Lie algebra. Then $\Lg$ admits a non-trivial CPA-structure.
\end{cor}

\section{CPA-structures on parabolic subalgebras of semsimple Lie algebras}

For this section we will assume that all Lie algebras are complex. 
The following construction yields a class of CPA-structures which is important
for the case of parabolic subalgebras of semisimple Lie algebras.

\begin{prop}\label{4.1}
Let $I$ be an ideal in $\Lg$ with center $\Lz=Z(I)$ such that $\Lg/I$ is abelian. Then every $1$-cocycle 
$f\in Z^1(\Lg/I,\Lz)$ defines an associative nil-inner CPA-structure on $\Lg$ by
\[
x\cdot y=[f(\ov{x}),y]
\]
for all $x,y\in \Lg$.
\end{prop}

\begin{proof}
Note that $\Lz$ is a characteristic ideal of $I$, and hence an ideal of $\Lg$. Therefore $\Lg$ acts on
$\Lz$ by the adjoint action $x\kringel z=[x,z]$ for all $x\in \Lg$ and $z\in \Lz$. Since $I$ acts trivially on
$\Lz$ we obtain an induced action on the quotient $\Lg/I$ on $\Lz$ by $\ov{x}\kringel z=[x,z]$. Now
$Z^1(\Lg/I,\Lz)$ consists of linear maps $f\colon \Lg/I\ra \Lz$ satisfying
\[
f([\ov{x},\ov{y}]=-\ov{y}\kringel f(\ov{x})+\ov{x}\kringel f(\ov{y})
\]
Since $\Lg/I$ is abelian, the condition reduces to $[f(\ov{x}),y]=[f(\ov{y}),x]$ for all $x,y \in \Lg$.
We claim that $x\cdot y=[f(\ov{x}),y]$ satisfies the axioms \eqref{com4},  \eqref{com5},  \eqref{com6}, 
of a CPA-structure. By the last remark we have $x\cdot y=y\cdot x$, so that \eqref{com4} is satisfied.
All products $x\cdot (y\cdot z)=[f(\ov{x}),[f(\ov{y}),z]]\subseteq [\Lz,\Lz]=0$ are zero, so that the 
CPA-product is nil-inner and associative. Furthermore we have $[x,y]\cdot z=[f([\ov{x},\ov{y}]),z]=0$, 
and hence \eqref{com5} is satisfied. Finally the Jacobi identity for the bracket on $I$ implies that
\begin{align*}
x\cdot[y,z] & = [f(\ov{x}),[y,z]] \\
            & = [[f(\ov{x}),y],z]+[y,[f(\ov{x}),z]]\\
            & = [x\cdot y,z]+[y,x\cdot z].
\end{align*}
Hence also \eqref{com6} is satisfied.
\end{proof}

\begin{rem}
Proposition $\ref{4.1}$ once more implies that every non-trivial solvable Lie algebra $\Lg$ with trivial center admits 
a non-trivial CPA-structure. In fact, take $I=[\Lg,\Lg]$, so that the quotient $\Lg/I$ is abelian. By assumption
$I\neq 0$, and $I$ is nilpotent, so that $\Lz:=Z(I)$ is non-trivial. Since $\Lg$ has trivial center we have 
$[[\Lz,\Lg],\Lg]]\neq 0$, so that $x\cdot y:=[[z,x],y]$ defines a non-trivial CPA-product for any $z\neq 0$ in $\Lz$.
\end{rem}

\begin{defi}
Denote by $\fix (\Lg)$ the Lie ideal of $\Lg$ generated by the set
\[
\{x\in \Lg\mid \ad(y)x=x \text{ for some } y\in \Lg\}.
\]
\end{defi}

We have $\fix(\Lg)=0$ if and only if $\Lg$ is nilpotent by Engel's theorem. For the other extreme we have
the following result:

\begin{lem}\label{4.4}
We have $\fix(\Lg)=\Lg$ if and only if $\Lg$ is perfect.
\end{lem}

\begin{proof}
Since a perfect Lie algebra $\Lg$ is generated by any of its Levi subalgebras, we
may assume that $\Lg$ is semisimple. Let $\{e_i,f_i,h_i\mid 1\le i\le k\}$ be the Chevalley-Serre generators
of $\Lg$. We have $[\frac{1}{2}h_i,e_i]=e_i$ and $[-\frac{1}{2}h_i,f_i]=f_i$ for all $i$, so that $e_i,f_i\in \fix(\Lg)$
for all $i$. Hence we also have $h_i=[e_i,f_i]\in \fix(\Lg)$, so that $\Lg\subseteq \fix(\Lg)\subseteq \Lg$.
Conversely, if $\fix(\Lg)=\Lg$, then $\fix(\Lg)\subseteq [\Lg,\Lg]$ implies that $[\Lg,\Lg]=\Lg$.
\end{proof}

\begin{lem}\label{4.5}
Let $\Lp$ be a parabolic subalgebra of a semisimple Lie algebra $\Lg$, and $\Ls$ be a Levi subalgebra
of $\Lp$. Then $\fix(\Lp)=[\Lp,\Lp]=\Ls\ltimes \nil (\Lp)$.
\end{lem}

\begin{proof}
Let $\Lb$ be a standard Borel subalgebra of $\Lg$ with standard generators $h_i,x_i$, $i=1,\ldots ,k$.
Since $[\frac{1}{2}h_i,x_i]=x_i$ we obtain $x_i\in \fix(\Lp)$ for all $i$. We have $ \nil(\Lb)\subseteq \fix (\Lb)$,
since $\nil(\Lb)$ is generated by the $x_i$. On the other hand, $\fix(\Lb)\subseteq [\Lb,\Lb]\subseteq \nil(\Lb)$
yields $\nil(\Lb)=\fix(\Lb)$. Furthermore we have that $\nil(\Lp)\subseteq \nil(\Lb)=\fix(\Lb)\subseteq \fix(\Lp)$. 
By Lemma $\ref{4.4}$ we have $\Ls\subseteq \fix(\Lp)$, so that $\Ls\ltimes \nil(\Lp)\subseteq 
\fix(\Lp)$. Conversely we have $\fix(\Lp)\subseteq [\Lp,\Lp]\subseteq \Ls\ltimes \nil(\Lp)$. It is well-known that
$\Ls\ltimes \nil (\Lp)=[\Lp,\Lp]$ for parabolic subalgebras of semisimple Lie algebras.
\end{proof}

\begin{lem}\label{4.6}
Let $x\cdot y=[\phi(x),y]$ be a nil-inner CPA-structure on $\Lg$. Then $\phi(\Lg)\subseteq \nil(\Lg)$ and 
$\fix(\Lg)\subseteq \ker(\phi)$.
\end{lem}

\begin{proof}
We already have seen that $x\cdot y=y\cdot x$ is equivalent to the identity $[\phi(x),y]=-[x,\phi(y)]$.
This yields $\ad (\phi(x))(y)=-\ad(x)^m(\phi^{2^m-1}(y))$ by induction on $m$. Since $\phi$ is nilpotent,
this implies  $\phi(\Lg)\subseteq \nil(\Lg)$. Now let $x,y\in \Lg$ with $[y,x]=x$.
Then we have
\[
0=\ad(\phi(y))^m(\phi(x))=\phi(\ad(y)^m(x))=\phi(x)
\]
for sufficiently large $m$. Since $\phi$ is a homomorphism, it also vanishes on the Lie ideal generated by such
$x$. This means $\phi(\fix(\Lg))=0$.
\end{proof}
 
We can now give a description of CPA-structures on parabolic subalgebras $\Lp$ of simple
Lie algebras. There are two cases, namely that the Borel subalgebra of $\Lp$ is metabelian, or not.
The metabelian case is as follows.

\begin{lem}
Let $\Ls$ be a simple Lie algebra and $\Lp$ a parabolic subalgebra of $\Ls$. Then $\Lp$ is metabelian if
and only if $\Ls$ is of type $A_1$ and $\Lp$ a Borel subalgebra. 
\end{lem}

\begin{proof}
Suppose that $\Lp$ is a Borel subalgebra of $A_1$. Then $\Lp$ is metabelian.
Conversely suppose that $\Lp$ is metabelian. Hence $\Lp$ is a solvable parabolic subalgebra
of $\Ls$, hence a Borel subalgebra, which we denote by $\Lb$ now.
Denote by $\Ln$ the nilradical of $\Lb$. Since $[\Lb,\Lb]=\Ln$ it is enough to show that
$\Ln$ is abelian if and only if $\Ls$ is of type $A_1$. However we have $\dim Z(\Ln)=1$ for
all simple Lie algebras $\Ls$, see \cite{SNW} section $4$, so that $\Ln$ is abelian if and only 
if $\Ln$ is $1$-dimensional. This is true if and only if $\Ls$ is of type $A_1$, see table $2$
in \cite{SNW}, which gives the dimensions of the nilradicals of $\Lb$ for all simple Lie algebras.
\end{proof}

The remaining case, where $\Lp$ is not metabelian, is as follows.

\begin{thm}\label{4.7}
Let $\Lp$ be a parabolic subalgebra of a simple Lie algebra $\Lg$, and suppose that $\Lp$ is not metabelian.
Denote by $\Lz$ the center of the ideal $I=[\Lp,\Lp]$.
Then there is a bijective correspondence between CPA-products on $\Lp$ and elements $z\in \Lz$, given by
\[
x\cdot y=[[z,x],y].
\]
\end{thm}

\begin{proof}
Since parabolic subalgebras of semisimple Lie algebras are complete, any CPA-structure on $\Lp$ is inner
by Corollary $\ref{2.9}$. In fact, any CPA-structure on $\Lp$ is nil-inner by Theorem $\ref{2.14}$, since
$\Lp$ is indecomposable and not metabelian. Writing $x\cdot y=[\phi(x),y]$ we have
$\phi(I)\subseteq \phi(\fix(\Lp))=0$ by Lemma $\ref{4.5}$ and Lemma $\ref{4.6}$. The identity $x\cdot y=y\cdot x$
yields $[\phi(\Lp),I]=[\phi(I),\Lp]=0$. Lemma $\ref{4.6}$ gives $\phi(\Lp)\subseteq \nil(\Lp)\subseteq I$, so that
$\phi(\Lp)\subseteq \Lz$. Hence $\phi$ may be identified with its restriction $\phi\colon \Lp\ra \Lz$.
By Lemma  $\ref{4.6}$ we obtain $I\subseteq \ker(\phi)$, so that $\phi$ projects to a quotient map
$f\colon \Lp/I\ra \Lz$, $\ov{x}\mapsto \phi(x)$. By commutativity of the product  we obtain
$[f(\ov{x}),y]=[f(\ov{y}),x]$ for all $x,y\in \Lp$. Since $\Lp/I$ is abelian and $f(\Lp)\subseteq \Lz$ this implies
$f\in Z^1/\Lp/I,\Lz)$ and $x\cdot y=[f(\ov{x}),y]$. \\
Conversely, every $1$-cocycle $f\in Z^1(\Lp/I,\Lz)$ defines a nil-inner CPA product on $\Lp$ by Proposition 
$\ref{4.1}$. Since $\Lp/I$ is abelian and $Z(\Lp)=0$ we have $H^0(\Lp/I,\Lz)\subseteq Z(\Lp)=0$. This implies
$H^n(\Lp/I,\Lz)=0$ for all $n\ge 0$, and in particular for $n=1$ we obtain $x\cdot y=[[z,x],y]$ for $z\in\Lz$.
Hence we have a bijective correspondence between CPA-products on $\Lp$ and elements $z\in \Lz$.
\end{proof}

We can now review Example $\ref{3.6}$.

\begin{ex}
The $6$-dimensional parabolic subalgebra $\Lg$ of $\Ls\Ll_3(\C)$ given in Example $\ref{3.6}$ admits 
only the trivial CPA-product.
\end{ex}

With the notations of Theorem $\ref{4.7}$ we have $\Ls=\langle e_1,e_3,e_5\rangle$ acting on
$\nil(\Lg)=\langle e_2,e_4\rangle$ by the irreducible action of dimension $2$. In particular we have 
$Z(I)=Z(\Ls\ltimes \nil(\Lg))=0$, so that all CPA-products on $\Lg$ vanish.

We can now describe explicitly all CPA-products on parabolic subalgebras of simple
Lie algebras $\Ls$. We may assume that $\Ls$ has rank at least $2$. 
For the rank one case see Example $\ref{2.16}$. 
We demonstrate the result for standard Borel subalgebras $\Lb$ of simple Lie algebras type $A_n$.
Let  $h_i,x_i$,  $i=1,\ldots ,k$ be  the standard generators of $\Lb$, and let $z$ be a generator 
of $Z(\nil(\Lb))$. 

\begin{prop}
Suppose that $\Ls=\Ls\Ll_{k+1}(\C)$ with $k\ge 2$, and $\Lb$ a standard Borel subalgebra of $\Ls$. 
Then we have $[h_1,z]=[h_k,z]=z$ and $[h_i,z]=0$ for all 
$i\neq 1,k$. All CPA-products on $\Lb$ are scalar multiples of the following product
\[
h_1\cdot h_1=h_k\cdot h_k=h_1\cdot h_k=h_k\cdot h_1=z.
\]
\end{prop}

We would like to extend Theorem $\ref{4.7}$ to parabolic subalgebras of semisimple Lie algebras $\Ls$.
Since parabolic subalgebras of semisimple Lie algebras are complete, we first study the case of
complete Lie algebras. The following definition is given in \cite{MEN}.

\begin{defi}
A complete Lie algebra $\Lg$ is called {\it simply-complete}, if no non-trivial ideal in
$\Lg$ is complete. 
\end{defi}

Meng \cite{MEN} showed that every complex complete Lie algebra $\Lg$ decomposes into a direct
sum of simply-complete ideals, and this decomposition is unique up to permutation of the ideals.

\begin{prop}
Let $\Lq_1,\ldots ,\Lq_n$ be simply-complete Lie algebras, each with a CPA-product. Then the direct 
Lie algebra sum admits a CPA-product, which is given componentwise: 
\[
(x_1,\ldots,x_n)\cdot (y_1,\ldots,y_n)=(x_1\cdot y_1,\ldots,x_n\cdot y_n).
\]
Conversely, for any complete Lie algebra $\Lq=\Lq_1\oplus \cdots \oplus \Lq_n$
with simply-complete ideals $\Lq_i$, any CPA-product on $\Lq$ is given as above.
\end{prop}

\begin{proof}
The first part is clear. For the second part we need only show that $\Lq_i\cdot \Lq_j\subseteq 
\Lq_i\cap \Lq_j$. Because all derivations of $\Lq$ are inner, we have $\Lq_i\cdot \Lq_j\subseteq 
\Der(\Lq)(\Lq_j)\subseteq \Lq_i$, and because the CPA-product is commutative also  $\Lq_i\cdot \Lq_j\subseteq \Lq_j$.
\end{proof}

We think that this will lead to a classification of CPA-products on parabolic subalgebras of semisimple
Lie algebras. Furthermore one might also wish to extend  the results to parabolic subalgebras of 
{\em reductive} Lie algebras. Let $\Lq$ be a parabolic subalgebra of a reductive Lie algebra $\Lg$. 
Then $\Der(\Lq)=\LL \oplus \ad (\Lq)$ as a Lie algebra direct sum, where $\LL$ is the set of all linear 
transformations $D\colon \Lq\ra Z(\Lq)$ such that $D([\Lq,\Lq])=0$, see \cite{BRH}. Furthermore we have 
$Z(\Lq)=Z(\Lg)$. However, the situation here is more complicated than before.


\begin{thebibliography}{99}

\bibitem{BRH} D. Brice, H. Huang: {\it On derivations of parabolic Lie algebras}.
arXiv: 1504.08286v3 (2015).

\bibitem{BU5} D. Burde: {\it Affine structures on nilmanifolds}. International Journal of
Mathematics \textbf{7} (1996), no. 5, 599--616.

\bibitem{BU19} D. Burde, K. Dekimpe, S. Deschamps: {\it The Auslander conjecture for NIL-affine
crystallographic groups}. Mathematische Annalen \textbf{332} (2005), no. 1, 161--176.

\bibitem{BU24} D. Burde: {\it Left-symmetric algebras, or pre-Lie algebras in geometry
and physics}. Central European Journal of Mathematics \textbf{4} (2006), no. 3, 323--357.

\bibitem{BU33} D. Burde, K. Dekimpe and S. Deschamps: {\it Affine actions on nilpotent
Lie groups}. Forum Math.\ \textbf{21} (2009), no. 5, 921--934.

\bibitem{BU34} D. Burde, K. Dekimpe and S. Deschamps: {\it LR-algebras}.
Contemporary Mathematics \textbf{491} (2009), 125--140.

\bibitem{BU38} D. Burde, K. Dekimpe, K. Vercammen: {\it Complete LR-structures on solvable
Lie algebras}. Journal of group theory  \textbf{13}, no. $5$, 703--719 (2010).

\bibitem{BU41} D. Burde, K. Dekimpe and K. Vercammen: {\it Affine actions on Lie groups
and post-Lie algebra structures}.
Linear Algebra and its Applications \textbf{437} (2012), no. 5, 1250--1263.

\bibitem{BU44} D. Burde, K. Dekimpe: {\it Post-Lie algebra structures and generalized
derivations of semisimple Lie algebras}.
Moscow Mathematical Journal, Vol. \textbf{13}, Issue 1, 1--18 (2013).

\bibitem{BU51} D. Burde, K. Dekimpe: {\it Post-Lie algebra structures on pairs of Lie algebras}.
Preprint $2015$.


\bibitem{HEL} J. Helmstetter: {\it Radical d'une alg\`ebre sym\'etrique
a gauche}. Ann.\ Inst.\ Fourier \textbf{29} (1979), 17-35.

\bibitem{KIM} H. Kim: {\it Complete left-invariant affine structures on nilpotent Lie groups}.
J. Differential Geom.  \textbf{24} (1986), no. 3, 373--394.

\bibitem{LEL} G. F. Leger, E. M. Luks: {\it Generalized Derivations of Lie Algebras}. 
J.\ Algebra \textbf{228}, (2000),  no. 1, 165-–203. 

\bibitem{LOD} J.-L. Loday: {\it Generalized bialgebras and triples of operads}.
Astérisque  No. \textbf{320}  (2008), 116 pp. 

\bibitem{MEN} D. Meng: {\it Some results on complete lie algebras}.
Communications in Algebra \textbf{22}, no. 13 (1994), 5457--5507.

\bibitem{SNW}  L. Snobl, P. Winternitz: {\it Solvable Lie algebras with Borel nilradicals}.
J.\ Phys.\ A 45 (2012), no. 9, 095202, 18 pp.

\bibitem{SEG} D. Segal: {\it The structure of complete left-symmetric algebras}.
Math.\ Ann.\ \textbf{293} (1992), 569--578.

\bibitem{VAL} B. Vallette: {\it Homology of generalized partition posets}.
J.\ Pure and Applied Algebra \textbf{208}, no. 2 (2007), 699--725. 

\end{thebibliography}
\end{document}